\documentclass{amsproc}%
\usepackage{amsfonts}
\usepackage{amsmath}
\usepackage{amssymb}
\usepackage{graphicx}%
\providecommand{\U}[1]{\protect \rule{.1in}{.1in}}
\theoremstyle{plain}

\newtheorem{lemma}{Lemma}

\newtheorem{proposition}{Proposition}

\newtheorem{theorem}{Theorem}
\numberwithin{equation}{section}
\begin{document}
\title[Beurling Theorem]{Lebesgue and Hardy Spaces for Symmetric Norms II: \\A Vector-Valued Beurling Theorem}
\author{Yanni Chen}
\address{Department of Mathematics, University of New Hampshire, Durham, NH 03824, U.S.A.}
\email{yet2@wildcats.unh.edu}
\author{Don Hadwin}
\curraddr{Department of Mathematics, University of New Hampshire, Durham, NH 03824, U.S.A.}
\email{don@unh.edu}
\urladdr{http://euclid.unh.edu/\symbol{126}don}
\author{Ye Zhang}
\address{Department of Mathematics, University of New Hampshire, Durham, NH 03824, U.S.A.}
\email{yjg2@unh.edu}
\thanks{This work supported in part by a UNH Dissertation Fellowship, $~$the Eric
Nordgren Research Fellowship, and a grant from the Simons Foundation.}

\begin{abstract}
Suppose $\alpha$ is a rotationally symmetric norm on $L^{\infty}\left(
\mathbb{T}\right)  $ and $\beta$ is a "nice" norm on $L^{\infty}\left(
\Omega,\mu \right)  $ where $\mu$ is a $\sigma$-finite measure on $\Omega$. We
prove a version of Beurling's invariant subspace theorem for the space
$L^{\beta}\left(  \mu,H^{\alpha}\right)  .$ Our proof uses the version of
Beurling's theorem on $H^{\alpha}\left(  \mathbb{T}\right)  $ in \cite{Chen}
and measurable cross-section techniques. Our result significantly extends a
result of H. Rezaei, S. Talebzadeh, and D. Y. Shin \cite{RTS}.

\end{abstract}
\maketitle

\section{Introduction}

Among the classical results that exemplify strong links between complex
analysis and operator theory, one of the most prominent places is occupied by
the description of all shift-invariant subspaces in the Hardy spaces and its
numerous generalizations (see \cite{ARS}, \cite{PRH}, \cite{Sai} and
\cite{YH}). The original statement concerning the space $H^{2}$ of functions
on the unit disk $\mathbb{D}$ was proved by A. Beurling \cite{Arv},
\cite{Helson}, and was later extended to $H^{p}$ classes by T. P. Srinivasan
\cite{Sr}. Further generalizations covering the vector-valued Hardy spaces
(attributed to P. Lax, H. Helson, D. Lowdenslager, P. R. Halmos, J. Rovnyak
and L. de Branges, but usually referred to as the Halmos-Beurling-Lax Theorem)
were used to obtain a functional model for a class of subnormal operators. In
\cite{Chen}, the first author extended the $H^{p}$ result by replacing the
$p$-norms with continuous rotationally symmetric norms $\alpha$ on $L^{\infty
}\left(  m\right)  $, where $m$ is Haar measure on the unit circle
$\mathbb{T}$, and defining $H^{\alpha}$ to be the $\alpha$-completion of the
set of polynomials. Recently, H. Rezaei, S. Talebzadeh, D. Y. Shin \cite{RTS}
described certain shift-invariant subspaces of $H^{2}\left(  \mathbb{D}%
,\mathcal{H}\right)  $ where $\mathcal{H}$ is a separable Hilbert space and
$\mathbb{D}$ is the open unit disk in the complex plane $\mathbb{C}$. In this
paper, we prove a very general version of Beurling's theorem that includes the
results in \cite{RTS} and \cite[Theorem 7.8]{Chen} as very special cases. A
key ingredient is the theory of measurable cross-sections \cite{Arv}. \bigskip

\section{Preliminaries}

For $1\leq p<\infty,$ the Hardy space $H^{p}:=H^{p}(\mathbb{D})$ is the space
of all holomorphic functions $f:\mathbb{D}\rightarrow \mathbb{C}$ for which
\[
\Vert f\Vert_{H^{p}}:=\lim_{r\rightarrow1}(\frac{1}{2\pi}\int_{0}^{2\pi
}|f(re^{i\theta})|^{p}d\theta)^{\frac{1}{p}}<\infty.
\]
An \emph{inner function} $\phi \in H^{2}$ is a bounded analytic function on
$\mathbb{D}$ with non-tangential boundary values of modulus $1$ a.e. $m$,
where $m$ is normalized arc-length measure on the unit circle $\mathbb{T}$.

Suppose $\left(  \Omega,\mu \right)  $ is a $\sigma$-finite measure space such
that $L^{1}\left(  \mu \right)  $ is separable. Let $L_{0}^{\infty}\left(
\mu \right)  $ denote the set of (equivalence classes of) bounded measurable
functions $f:\Omega \rightarrow \mathbb{C}$ such that $\mu \left(  f^{-1}\left(
\mathbb{C}\backslash \left \{  0\right \}  \right)  \right)  <\infty$, and let
$\beta$ be a norm on $L_{0}^{\infty}\left(  \mu \right)  $ such that

\begin{enumerate}
\item $\beta \left(  f\right)  =\beta \left(  \left \vert f\right \vert \right)
,$

\item $\lim_{\mu \left(  E\right)  \rightarrow0}\beta \left(  \chi_{E}\right)
=0$,

\item $\beta \left(  f_{n}\right)  \rightarrow0$ implies $\chi_{E}%
f_{n}\rightarrow0$ in measure for each $E$ with $\mu \left(  E\right)  <\infty$.
\end{enumerate}

Examples of such norms are the norms $\left \Vert \cdot \right \Vert _{p}$ when
$1\leq p<\infty$. We define $L^{\beta}\left(  \mu \right)  $ to be the
completion of $L_{0}^{\infty}\left(  \mu \right)  $ with respect to $\beta$. At
this point we do not know that the elements of $L^{\beta}\left(  \mu \right)  $
can be represented as measurable functions. This follows from part $\left(
5\right)  $ of the following lemma, which also includes some basic facts about
such norms $\beta$.

\begin{lemma}
\label{1}The following statements are true for $\left(  \Omega,\mu \right)  $
and $\beta$ as above.

\begin{enumerate}
\item If $\left \vert f\right \vert \leq \left \vert g\right \vert $, then
$\beta \left(  f\right)  \leq \beta \left(  g\right)  $ whenever $f,g\in
L_{0}^{\infty}\left(  \mu \right)  ;$

\item $\beta \left(  wf\right)  \leq \left \Vert w\right \Vert _{\infty}%
\beta \left(  f\right)  $ whenever $f\in L_{0}^{\infty}\left(  \mu \right)  $
and $w\in L^{\infty}\left(  \mu \right)  ;$

\item The multiplication $wf=fw$ can be extended from part $\left(  2\right)
$ to the case where $f\in L^{\beta}\left(  \mu \right)  $ and $w\in L^{\infty
}\left(  \mu \right)  $, so that $\beta \left(  wf\right)  \leq \left \Vert
w\right \Vert _{\infty}\beta \left(  f\right)  $ still holds, i.e., $L^{\beta
}\left(  \mu \right)  $ is an $L^{\infty}\left(  \mu \right)  $-bimodule$;$

\item If $\left \{  E_{n}\right \}  $ is a sequence of measurable sets such that
$\mu \left(  E_{n}\cap F\right)  \rightarrow0$ for every $F$ with $\mu \left(
F\right)  <\infty$, then $\beta \left(  \chi_{E_{n}}f\right)  \rightarrow0$ for
every $f\in L^{\beta}\left(  \mu \right)  $;

\item If $\left \{  f_{n}\right \}  $ is a $\beta$-Cauchy sequence in
$L_{0}^{\infty}\left(  \mu \right)  $ and $\chi_{E}f_{n}\rightarrow0$ in
measure for each $E$ with $\mu \left(  E\right)  <\infty$, then $\beta \left(
f_{n}\right)  \rightarrow0;$

\item If $\left \vert f\right \vert \leq \left \vert g\right \vert $ and $g\in
L^{\beta}\left(  \mu \right)  $, then $f\in L^{\beta}\left(  \mu \right)  $ and
$\beta \left(  f\right)  \leq \beta \left(  g\right)  ;$

\item If $h\in L^{\beta}\left(  \mu \right)  ,$ $\left \{  f_{n}\right \}  $ is a
sequence, $\left \vert f_{n}\right \vert \leq h$ for $n\geq1$ and $\chi
_{E}\left \vert f_{n}-f\right \vert \rightarrow0$ in measure for every
$E\subset \Omega$ with $\mu \left(  E\right)  <\infty$, then $f\in L^{\beta
}\left(  \mu \right)  $ and $\beta \left(  f_{n}-f\right)  \rightarrow0;$

\item $L^{\beta}\left(  \mu \right)  $ is a separable Banach space.
\end{enumerate}
\end{lemma}

\begin{proof}
$\left(  1\right)  .$ If $\left \vert f\right \vert \leq \left \vert g\right \vert
$, then there are two measurable functions $u,v$ with $\left \vert u\right \vert
=\left \vert v\right \vert =1$ and $f=g\left(  u+v\right)  /2$, which implies
$\beta \left(  f\right)  \leq \left[  \beta \left(  \left \vert ug\right \vert
\right)  +\beta \left(  \left \vert vg\right \vert \right)  \right]
/2=\beta \left(  \left \vert g\right \vert \right)  =\beta \left(  g\right)  .$

$\left(  2\right)  $. Since $\left \vert wf\right \vert \leq \left \Vert
w\right \Vert _{\infty}\left \vert f\right \vert ,$ it follows from part $\left(
1\right)  $ that
\[
\beta \left(  wf\right)  \leq \beta \left(  \left \Vert w\right \Vert _{\infty
}\left \vert f\right \vert \right)  =\left \Vert w\right \Vert _{\infty}%
\beta \left(  f\right)  .
\]

$\left(  3\right)  $. The mapping $M_{w}:$ $L_{0}^{\infty}\left(  \mu \right)
\rightarrow L_{0}^{\infty}\left(  \mu \right)  $ defined by $M_{w}f=wf=fw$ is
bounded on $L_{0}^{\infty}\left(  \mu \right)  $ equipped with the norm
$\beta,$ so it has a unique bounded linear extension to the completion
$L^{\beta}\left(  \mu \right)  .$

$\left(  4\right)  .$ Suppose $\left \{  E_{n}\right \}  $ is a sequence of
measurable sets such that $\mu \left(  E_{n}\cap F\right)  \rightarrow0$ for
every $F$ with $\mu \left(  F\right)  <\infty$. Define $T_{n}:L^{\beta}\left(
\mu \right)  \rightarrow L^{\beta}\left(  \mu \right)  $ by $T_{n}f=\chi_{E_{n}%
}f$. The set $\mathcal{E}=\left \{  f\in L^{\beta}\left(  \mu \right)
:\beta \left(  T_{n}f\right)  \rightarrow0\right \}  $ is a closed linear
subspace. If $f\in L_{0}^{\infty}\left(  \mu \right)  $, then there is a set
$F$ with $\mu \left(  F\right)  <\infty$ such that $f=\chi_{F}f,$ so%
\[
\beta \left(  T_{n}f\right)  \leq \left \Vert f\right \Vert _{\infty}\beta \left(
\chi_{E_{n}\cap F}\right)  \rightarrow0
\]
since $\mu \left(  E_{n}\cap F\right)  \rightarrow0$. Hence
\[
L^{\beta}\left(  \mu \right)  \subset L_{0}^{\infty}\left(  \mu \right)
^{-\beta}\subset \mathcal{E}.
\]

$\left(  5\right)  $. By definition there is an $f\in L^{\beta}\left(
\mu \right)  $ such that $\beta \left(  f_{n}-f\right)  \rightarrow0.$ Choose
$M$ so that $\sup_{n\geq1}\beta \left(  f_{n}\right)  <M<\infty$. At this point
we do not know that $f$ is a measurable function. Suppose $\mu \left(
F\right)  <\infty$ and $\varepsilon>0$. It easily follows from part $\left(
4\right)  $ that there is a $\delta>0$ such that when $E\subset F$ and
$\mu \left(  E\right)  <\delta$, we have $\beta \left(  \chi_{E}f\right)
<\varepsilon/4$. There is an $N\in \mathbb{N}$ such that $n\geq
N\Longrightarrow \beta \left(  f_{n}-f\right)  <\varepsilon/4$. Thus $n\geq N$
and $E\subset F$ and $\mu \left(  E\right)  <\delta$ implies%
\[
\beta \left(  \chi_{E}f_{n}\right)  \leq \beta \left(  \chi_{E}\left(
f_{n}-f\right)  \right)  +\beta \left(  \chi_{E}f\right)  <\varepsilon
/4+\varepsilon/4=\varepsilon/2.
\]
Since $f_{n}\chi_{F}\rightarrow0$ in measure, there is an $N_{1}>N$ such that
if $n\geq N_{1}$ and if $F_{n}=\left \{  x\in F:\left \vert f_{n}\left(
x\right)  \right \vert \geq \varepsilon/4\beta \left(  \chi_{F}\right)  \right \}
,$ then $\mu \left(  F_{n}\right)  <\delta$. Thus $n\geq N_{1}$ implies
\[
\beta \left(  \chi_{F}f_{n}\right)  \leq \beta \left(  \chi_{F_{n}}f_{n}\right)
+\left(  \varepsilon/4\beta \left(  \chi_{F}\right)  \right)  \beta \left(
\chi_{F}\right)  <\varepsilon.
\]
So $\chi_{F}f=0$ for every $F\subset \Omega$ with $\mu \left(  F\right)
<\infty$. It follows from part $\left(  4\right)  $ and the fact that $\mu$ is
$\sigma$-finite that $\beta \left(  f\right)  =\lim \beta \left(  f_{n}\right)
=0.$

$\left(  6\right)  $. Suppose $\left \vert f\right \vert \leq \left \vert
g\right \vert $ and $g\in L^{\beta}\left(  \mu \right)  $. We know from part
$\left(  5\right)  $ that $g$ is a measurable function so there is a $w\in
L^{\infty}\left(  \mu \right)  $ such that $f=wg\in L^{\beta}\left(
\mu \right)  $ (by part $\left(  2\right)  $).

$\left(  7\right)  $. Assume the hypothesis of part $\left(  7\right)  $ holds
and suppose $\varepsilon>0.$ Since $\mu$ is $\sigma$-finite, it follows that
there is a subsequence $\left \{  f_{n_{k}}\right \}  $ that converges to $f$
a.e. $\left(  \mu \right)  $. Hence $\left \vert f\right \vert \leq h$ a.e.
$\left(  \mu \right)  $. Then, by part $\left(  4\right)  $, there is a
$\delta$ such that $E\subset \Omega$ and $\mu \left(  E\right)  <\delta$ implies
$\beta \left(  \chi_{E}h\right)  <\varepsilon/5$. By part $\left(  6\right)  $
it follows that if $\mu \left(  E\right)  <\delta,$ then $\beta \left(  \chi
_{E}f\right)  ,\beta \left(  \chi_{E}f_{n}\right)  \leq \beta \left(  \chi
_{E}h\right)  <\varepsilon/5$ for every $n\geq1$. We also know from part
$\left(  4\right)  $ that there is a set $F$ with $\mu \left(  F\right)
<\infty$ such that $\beta \left(  \left(  1-\chi_{F}\right)  h\right)
<\varepsilon/5$, which implies $\beta \left(  \left(  1-\chi_{F}\right)
f\right)  ,\beta \left(  \left(  1-\chi_{F}\right)  f_{n}\right)
<\varepsilon/5$ for every $n\geq1$. But $f_{n}\chi_{F}\rightarrow f\chi_{F}$
in measure, so that if $E_{n}=\left \{  \omega \in F:\left \vert f_{n}\left(
\omega \right)  -f\left(  \omega \right)  \right \vert \geq \varepsilon
/5\beta \left(  \chi_{F}\right)  \right \}  $, then $\mu \left(  E_{n}\right)
\rightarrow0.$ Thus there is an $N\in \mathbb{N}$ such that, $n\geq N$ implies
$\mu \left(  E_{n}\right)  <\delta$, which implies%
\[
\beta \left(  f_{n}-f\right)  \leq
\]%
\[
\beta \left(  f_{n}\chi_{E_{n}}\right)  +\beta \left(  f\chi_{E_{n}}\right)
+\beta \left(  \left(  1-\chi_{F}\right)  f\right)  +\beta \left(  \left(
1-\chi_{F}\right)  f_{n}\right)  +\beta \left(  \chi_{F\backslash E_{n}}\left[
\varepsilon/5\beta \left(  \chi_{F}\right)  \right]  \right)  <\varepsilon.
\]

$\left(  8\right)  .$ It is clear that $L^{\beta}\left(  \mu \right)  $ is a
Banach space. Write $\Omega=\cup_{n\geq1}\Omega_{n}$ where $\left \{
\Omega_{n}\right \}  $ is an increasing sequence of sets with $\mu \left(
\Omega_{n}\right)  <\infty$ for $n\geq1$. Since $L^{1}\left(  \mu \right)  $ is
separable, we can find a countable subset $\mathcal{W}_{n}$ that is a
$\left \Vert \cdot \right \Vert _{1}$-dense subset of $\chi_{\Omega_{n}}\left \{
f\in L^{\infty}\left(  \mu \right)  :\left \vert f\right \vert \leq n\right \}  $.
It follows from part $\left(  7\right)  $ that $\mathcal{W}_{n}^{-\beta}$
contains $\chi_{\Omega_{n}}L_{0}^{\infty}\left(  \mu \right)  ,$ and if we let
$\mathcal{W=\cup}_{n\geq1}\mathcal{W}_{n}$, it follows from $\mathcal{W}%
^{-\beta}$ contains $L_{0}^{\infty}\left(  \mu \right)  ^{-\beta}=L^{\beta
}\left(  \mu \right)  $. Hence $L^{\beta}\left(  \mu \right)  $ is separable.
\end{proof}

\bigskip

Suppose $X$ is a separable Banach space and define
\[
L^{\beta}\left(  \mu,X\right)  =\left \{  f|f:\Omega \rightarrow X\text{ is
measurable and }\left \Vert \cdot \right \Vert \circ f\in L^{\beta}\left(
\mu \right)  \right \}  .
\]
If $f:\Omega \rightarrow X,$ define $\left \vert f\right \vert :\Omega
\rightarrow \lbrack0,\infty)$ by%
\[
\left \vert f\right \vert \left(  \omega \right)  =\left \Vert f\left(
\omega \right)  \right \Vert ,
\]
i.e., $\left \vert f\right \vert =\left \Vert \cdot \right \Vert \circ f$. It is
clear that if we define $\beta \left(  f\right)  =\beta \left(  \left \Vert
\cdot \right \Vert \circ f\right)  $, then $L^{\beta}\left(  \mu,X\right)  $ is
a Banach space. Moreover, $L^{\beta}\left(  \mu,X\right)  $ is an $L^{\infty
}\left(  \mu \right)  $-module if we define $\varphi f$ with $\varphi \in
L^{\infty}\left(  \mu \right)  $ and $f\in L^{\beta}\left(  \mu,X\right)  $ by%
\[
\left(  \varphi f\right)  \left(  \omega \right)  =\varphi \left(
\omega \right)  f\left(  \omega \right)  \in X.
\]
It is clear from part $\left(  1\right)  $ of Lemma \ref{1} that%
\[
\beta \left(  \varphi f\right)  \leq \left \Vert \varphi \right \Vert _{\infty
}\beta \left(  f\right)  .
\]
Since
\[
\left \vert \chi_{E}f\right \vert =\chi_{E}\left \vert f\right \vert
\]
for every $f\in L^{\beta}\left(  \mu,X\right)  $, it easily follows that parts
$\left(  2\right)  ,\left(  4\right)  ,\left(  6\right)  $ and $\left(
7\right)  $ in Lemma \ref{1} remain true if $f\in L^{\beta}\left(
\mu,X\right)  $.

\begin{lemma}
\label{2.2}If $X$ is separable, then $L^{\beta}\left(  \mu,X\right)  $ is separable.
\end{lemma}

\begin{proof}
It is well known \cite{DU} that $L^{1}\left(  \mu,X\right)  $ is separable. We
can imitate the proof of part $\left(  8\right)  $ of Lemma \ref{1} to get the
desired conclusion.
\end{proof}

\bigskip

Recall that $\alpha$ is a\emph{ rotationally symmetric norm} on $L^{\infty
}\left(  \mathbb{T}\right)  $ if

\begin{enumerate}
\item $\alpha \left(  1\right)  =1$,

\item $\alpha \left(  f\right)  =\alpha \left(  \left \vert f\right \vert \right)
,$

\item If $g\left(  z\right)  =f\left(  e^{i\theta}z\right)  $ ($\theta
\in \mathbb{R}$), then $\alpha \left(  f\right)  =\alpha \left(  g\right)  $.
\end{enumerate}

We say that a rotationally symmetric norm $\alpha$ is \emph{continuous} if
\[
\lim_{m\left(  E\right)  \rightarrow0}\alpha \left(  \chi_{E}\right)  =0.
\]

If $\alpha$ is a continuous rotationally symmetric norm on $L^{\infty}\left(
\mathbb{T}\right)  ,$ the space $H^{\alpha}$ is defined in the first part to
be the $\alpha$-closure of the linear span of $\left \{  1,z,z^{2}%
,\ldots \right \}  $. It is clear that $H^{\alpha}$ is separable and $H^{\infty
}\subset H^{\alpha}$. We also obtained a new version of Beurling's theorem in
the first part, namely, that if $M\neq \left \{  0\right \}  $ is a closed linear
subspace of $H^{\alpha}$ and $zM\subset M$, then $M=\varphi H^{\alpha}$ for
some inner function $\varphi \in H^{\infty}$. It follows from Lemma \ref{2.2}
that $L^{\beta}\left(  \mu,H^{\alpha}\right)  $ is separable.

We define $L^{\infty}\left(  \mu,H^{\infty}\right)  $ to be the set of
(equivalence classes) of bounded functions $\Phi:\Omega \rightarrow H^{\infty}$
that are weak*-measurable, and we define $\left \Vert \Phi \right \Vert _{\infty
}$ to be the essential supremum of $\left \Vert \cdot \right \Vert _{\infty}%
\circ \Phi$. It is clear that $L^{\infty}\left(  \mu,H^{\infty}\right)  $ is a
Banach algebra, and, since $H^{\alpha}$ is an $H^{\infty}$-module, we can make
$L^{\beta}\left(  \mu,H^{\alpha}\right)  $ an $L^{\infty}\left(  \mu
,H^{\infty}\right)  $-module by%
\[
\left(  \Phi f\right)  \left(  \omega \right)  =\Phi \left(  \omega \right)
f\left(  \omega \right)  .
\]
It is also clear%
\[
\beta \left(  \Phi f\right)  \leq \left \Vert \Phi \right \Vert _{\infty}%
\beta \left(  f\right)  .
\]

We can also define the \emph{shift operator} $S$ on $L^{\beta}\left(
\mu,H^{\alpha}\right)  $ by%
\[
\left(  \left(  Sf\right)  \left(  \omega \right)  \right)  \left(  z\right)
=z\left(  f\left(  \omega \right)  \right)  \left(  z\right)  .
\]
It is clear that $S$ is an isometry on $L^{\beta}\left(  \mu,H^{\alpha
}\right)  $ and that $S$ is an $L^{\infty}\left(  \mu,H^{\infty}\right)
$-module homomorphism, i.e.,%
\[
S\left(  \Phi f\right)  =\Phi \left(  Sf\right)
\]
whenever $f\in L^{\beta}\left(  \mu,H^{\alpha}\right)  $ and $\Phi \in
L^{\infty}\left(  \mu,H^{\infty}\right)  $.

\section{The Main Result}

Our main result (Theorem \ref{GBT}) is a generalization of the classical
Beurling theorem for $H^{p}$ \cite{Sr} and its extension to $H^{\alpha}$
\cite[Theorem 7.8]{Chen}. A key tool is a result on measurable cross-sections
taken from \cite{Arv}. A subset $A$ of a separable metric space $Y$ is
\emph{absolutely measurable} if $A$ is $\mu$-measurable for every $\sigma
$-finite Borel measure $\mu$ on $Y$. A function with domain $Y$ is
\emph{absolutely measurable} if the inverse image of every Borel set is
absolutely measurable.

\begin{lemma}
\label{arv}Suppose $E$ is a Borel subset of a complete separable metric space
and $Y$ is a separable metric space and $\pi:E\rightarrow Y$ is continuous.
Then $\pi \left(  E\right)  $ is an absolutely measurable subset of $Y$ and
there is an absolutely measurable function \newline$\rho:\pi \left(  E\right)
\rightarrow E$ such that $\left(  \pi \circ \rho \right)  \left(  y\right)  =y$
for every $y\in E$.
\end{lemma}

\begin{theorem}
\label{GBT}A closed linear subspace $M$ of $L^{\beta}\left(  \mu,H^{\alpha
}\right)  $ is an $L^{\infty}\left(  \mu \right)  $-submodule with $S\left(
M\right)  \subset M$ if and only if there is a $\Phi \in L^{\infty}\left(
\mu,H^{\infty}\right)  $ such that

\begin{enumerate}
\item For every $\omega \in \Omega$, we have $\Phi \left(  \omega \right)  =0$ or
$\Phi \left(  \omega \right)  $ is an inner function,

\item $M=\Phi L^{\beta}\left(  \mu,H^{\alpha}\right)  $.
\end{enumerate}
\end{theorem}

\begin{proof}
Suppose $\left(  1\right)  ,\left(  2\right)  $ are true. It is clear that
$\Phi L^{\beta}\left(  \mu,H^{\alpha}\right)  $ is a shift-invariant
$L^{\infty}\left(  \mu \right)  $-submodule. Let $E=\left \{  \omega \in
\Omega:\Phi \left(  \omega \right)  \neq0\right \}  $. Then $\chi_{E}L^{\beta
}\left(  \mu,H^{\alpha}\right)  $ is clearly closed and multiplication by
$\Phi$ is an isometry on $\chi_{E}L^{\beta}\left(  \mu,H^{\alpha}\right)  $.
Hence $\Phi L^{\beta}\left(  \mu,H^{\alpha}\right)  $ is closed.

Conversely, suppose $M$ is a shift invariant $L^{\infty}\left(  \mu \right)
$-submodule of $L^{\beta}\left(  \mu,H^{\alpha}\right)  $. Since $L^{\beta
}\left(  \mu,H^{\alpha}\right)  $ is separable, $M$ must be separable. We can
choose a countable subset $\mathcal{F}$ of $M$ such that $\mathcal{F}$ is
dense in $M,$ $S\left(  \mathcal{F}\right)  \subset \mathcal{F}$, and
$\mathcal{F}$ is a vector space over the field $\mathbb{Q}+i\mathbb{Q}$ of
complex-rational numbers. The elements of $\mathcal{F}$ are equivalence
classes, but we can choose actual functions to represent $\mathcal{F}$. Then,
for each $\omega \in \Omega$, define $M_{\omega}$ to be the $H^{\alpha}$-closure
of $\left \{  f\left(  \omega \right)  :f\in \mathcal{F}\right \}  $.

\textbf{Claim:} $M=\left \{  h\in L^{\beta}\left(  \mu,H^{\alpha}\right)
:h\left(  \omega \right)  \in M_{\omega}\text{ a.e. }\left(  \mu \right)
\right \}  .$

\textbf{Proof of Claim}: Suppose $h\in M$. Then there is a sequence $\left \{
f_{n}\right \}  $ in $\mathcal{F}$ such that $\beta \left(  f_{n}-h\right)  =$
$\beta \left(  \alpha \circ \left(  f_{n}-h\right)  \right)  \rightarrow0.$ We
know $\mu$ is $\sigma$-finite, so there is an increasing sequence $\left \{
\Omega_{k}\right \}  $ of sets of finite measure whose union is $\Omega.$ Since
$\left \{  \alpha \circ \left(  f_{n}-h\right)  \right \}  $ is a sequence in
$L^{\beta}\left(  \mu \right)  $, it follows from part $\left(  3\right)  $ in
the definition of $\beta$ and the fact that $\mu$ is $\sigma$-finite, that
$\chi_{\Omega_{k}}\alpha \circ \left(  f_{n}-h\right)  \rightarrow0$ in measure
for each $k\geq1$, which, via the Cantor diagonalization argument, implies
that there is a subsequence $\left \{  \alpha \circ \left(  f_{n_{k}}-h\right)
\right \}  $ that converges to $0$ a.e. $\left(  \mu \right)  $. Thus, for
almost every $\omega \in \Omega,$ we have $\alpha \left(  f_{n_{k}}\left(
\omega \right)  -h\left(  \omega \right)  \right)  \rightarrow0.$ Hence
$h\left(  \omega \right)  \in M_{\omega}$ a.e. $\left(  \mu \right)  $.

Conversely, suppose $h\in L^{\beta}\left(  \mu,H^{\alpha}\right)  $ and
$h\left(  \omega \right)  \in M_{\omega}$ a.e. $\left(  \mu \right)  $. By
redefining $h\left(  \omega \right)  =0$ on a set of measure $0$, we can assume
that $h\left(  \omega \right)  \in M_{\omega}$ for every $\omega \in \Omega$. Let
$X=H^{\alpha}\times%
{\displaystyle \prod_{n=1}^{\infty}}
H^{\alpha}\times \left(  0,1\right)  \times \mathbb{N}$ with the product
topology (giving $\left(  0,1\right)  $ the metric from the homeomorphism with
$\mathbb{R}$). Then $X$ is a complete separable metric space and the set $E$
of elements $\left(  g,g_{1},g_{2},\ldots,\varepsilon,n\right)  $ such that
$\alpha \left(  g-g_{n}\right)  \leq \varepsilon$ is closed in $X.$ Hence $E$ is
a complete separable metric space. Define $\pi:E\rightarrow Y=H^{\alpha}\times%
{\displaystyle \prod_{n=1}^{\infty}}
H^{\alpha}\times \left(  0,1\right)  $ by $\pi \left(  g,g_{1},g_{2}%
,\ldots,\varepsilon,n\right)  =\left(  g,g_{1},g_{2},\ldots,\varepsilon
\right)  .$ It follows from Lemma \ref{arv} that
\[
\pi \left(  E\right)  =\left \{  \left(  g,g_{1},g_{2},\ldots,\varepsilon
\right)  :\exists n\in \mathbb{N}\text{ with }\alpha \left(  g-g_{n}\right)
\leq \varepsilon \right \}
\]
is absolutely measurable and that there is an absolutely measurable
cross-section \newline$\rho:\pi \left(  E\right)  \rightarrow E$ such that
$\pi \left(  \rho \left(  y\right)  \right)  =y$ for every $y\in \pi \left(
E\right)  $. Suppose $\varepsilon>0$. Since $\mu$ is $\sigma$-finite there is
a function $u:\Omega \rightarrow \mathbb{R}$ such that $0<u<1$ and $\beta \left(
u\right)  \leq \varepsilon,$ i.e., if $\Omega$ is a disjoint union of sets
$\left \{  E_{n}\right \}  $ with finite measure, we can let%
\[
u=\varepsilon \sum_{n}\frac{\chi_{E_{n}}}{2^{n}\left(  1+\beta \left(
\chi_{E_{n}}\right)  \right)  }.
\]
We can write $\mathcal{F=}\left \{  f_{1},f_{2},\ldots \right \}  $ and define
$\Gamma:\Omega \rightarrow Y$ by%
\[
\Gamma \left(  \omega \right)  =\left(  h\left(  \omega \right)  ,f_{1}\left(
\omega \right)  ,f_{2}\left(  w\right)  ,\ldots,u\left(  \omega \right)
\right)  .
\]
Since $h\left(  \omega \right)  \in M_{\omega}=\left \{  f_{1}\left(
\omega \right)  ,f_{2}\left(  \omega \right)  ,\ldots \right \}  ^{-\alpha},$ it
follows that $\Gamma \left(  \Omega \right)  \subset \pi \left(  E\right)  $.

Since $\rho$ is absolutely measurable, $\rho \circ \Gamma$ is measurable, and if
we write%
\[
\left(  \rho \circ \Gamma \right)  \left(  \omega \right)  =\left(  \Gamma \left(
\omega \right)  ,n\left(  \omega \right)  \right)  ,
\]
we see that $n:\Omega \rightarrow \mathbb{N}$ is measurable and
\[
\alpha \left(  f_{n\left(  \omega \right)  }-h\left(  \omega \right)  \right)
\leq u\left(  \omega \right)
\]
for every $w\in \Omega.$ Let $G_{k}=\left \{  \omega \in \Omega:n\left(
\omega \right)  =k\right \}  .$ Then $\left \{  G_{k}:k\in \mathbb{N}\right \}  $
is a measurable partition of $\Omega$ and $f=\sum_{k=1}^{\infty}\chi_{G_{k}%
}f_{k}$ defines a measurable function from $\Omega$ to $H^{\alpha}$. Moreover,
if $\omega \in G_{k},$ then%
\[
\alpha \left(  f\left(  \omega \right)  -h\left(  \omega \right)  \right)
=\alpha \left(  f_{n\left(  \omega \right)  }-h\left(  \omega \right)  \right)
\leq u\left(  \omega \right)  .
\]
Hence $\alpha \circ \left(  f-h\right)  \in L^{\beta}\left(  \mu \right)  ,$ so
$f-h\in L^{\beta}\left(  \mu,H^{\alpha}\right)  $ and, by part $\left(
6\right)  $ of Lemma \ref{1},%
\[
\beta \left(  f-h\right)  \leq \beta \left(  u\right)  \leq \varepsilon.
\]
Thus $f=\left(  f-h\right)  +h\in L^{\beta}\left(  \mu,H^{\alpha}\right)  .$
Moreover, since $M$ is an $L^{\infty}\left(  \mu \right)  $-module, we have
$\sum_{k=1}^{N}\chi_{G_{k}}f_{k}\in M$ for each $N\in \mathbb{N}$ and
$f-\sum_{k=1}^{N}\chi_{G_{k}}f_{k}=\chi_{W_{N}}f$, where $W_{N}=\cup
_{k>N}G_{k}$. But
\[
\alpha \left(  \left(  \chi_{W_{N}}f\right)  \left(  \omega \right)  \right)
=\chi_{W_{N}}\left(  \omega \right)  \alpha \left(  f\left(  \omega \right)
\right)  \leq \left(  \alpha \circ f\right)  \left(  \omega \right)  .
\]
Since $\alpha \circ f\in L^{\beta}\left(  \mu \right)  $ and $\chi_{W_{N}%
}\left(  \omega \right)  \alpha \left(  f\left(  \omega \right)  \right)
\rightarrow0$ pointwise, it follows from the general dominated convergence
theorem, part $\left(  7\right)  $ of Lemma \ref{1} that
\[
\beta \left(  f-\sum_{k=1}^{N}\chi_{G_{k}}f_{k}\right)  \rightarrow0.
\]
Hence $f\in M$. But $M$ is closed, $\varepsilon>0$ was arbitrary and
$\beta \left(  f-h\right)  \leq \varepsilon,$ so $h\in M$. This proves the claim.

We next show that $zM_{\omega}\subset M_{\omega}$ for every $\omega \in \Omega$.
Indeed, recalling that $\mathcal{F=}\left \{  f_{1},f_{2},\ldots \right \}  $ and
$S\left(  \mathcal{F}\right)  \subset \mathcal{F}$, we see from the fact that
multiplication by $z$ is an isometry on $H^{\alpha}$ that%
\begin{align*}
zM_{\omega}  & =z\left \{  f_{1}\left(  \omega \right)  ,f_{2}\left(  w\right)
,\ldots \right \}  ^{-\alpha}\\
& =\left \{  zf_{1}\left(  \omega \right)  ,zf_{2}\left(  \omega \right)
,\ldots \right \}  ^{-\alpha}\\
& =\left \{  \left(  Sf_{1}\right)  \left(  \omega \right)  ,\left(
Sf_{2}\right)  \left(  \omega \right)  ,\ldots \right \}  ^{-\alpha}\\
& \subset \left \{  f_{1}\left(  \omega \right)  ,f_{2}\left(  w\right)
,\ldots \right \}  ^{-\alpha}=M_{\omega}.
\end{align*}
It follows from our version of Beurling's theorem \cite[Theorem 7.8]{Chen}
that either each $M_{\omega}=0$ or $M_{\omega}=\varphi H^{\alpha}$ for some
inner function $\varphi \in H^{\infty}$.

Let $\mathcal{I}$ be the set of inner functions in $H^{\infty}$. The algebra
$H^{\infty}$ can be viewed as an algebra of (multiplication) operators on
$H^{2}$, where the weak*-topology corresponds to the weak operator topology.
The set of inner functions is not weak operator closed, but it is closed in
the strong operator topology on $H^{2}$, since the set of inner functions
corresponds exactly to the operators in $H^{\infty}$ that are isometries.
Although the weak and strong operator topologies do not coincide, they
generate the same Borel sets. Hence, the set $\mathcal{I}$ with the strong
operator topology is a complete separable metric space and the Borel sets are
the same as the ones from the weak*-topology.

Let $\mathcal{X}=%
{\displaystyle \prod_{n=1}^{\infty}}
H^{\alpha}\times%
{\displaystyle \prod_{n=1}^{\infty}}
H^{\alpha}\times \mathcal{I}\times%
{\displaystyle \prod_{n=1}^{\infty}}
\mathbb{N}$ with the product topology with the strong operator topology. Let
$\mathcal{E}$ be the set of $\left(  g_{1},g_{2},\ldots,h_{1},h_{2}%
,\ldots,\varphi,n_{1},n_{2},\ldots \right)  $ in $X$ such that $g_{k}=\varphi
h_{k}$ and $\alpha \left(  \varphi-g_{n_{k}}\right)  <1/k$ for $1\leq k<\infty
$. Then $\mathcal{E}$ is a closed subset of the complete separable metric
space $\mathcal{X}.$ Define $\pi:\mathcal{X}\rightarrow \mathcal{Y}=%
{\displaystyle \prod_{n=1}^{\infty}}
H^{\alpha}$ by $\pi \left(  g_{1},g_{2},\ldots,h_{1},h_{2},\ldots,\varphi
,n_{1},n_{2},\ldots \right)  =\left(  g_{1},g_{2},\ldots \right)  $. Then
$\pi \left(  \mathcal{E}\right)  $ is the set of all $\left(  g_{1}%
,g_{2},\ldots \right)  \in \mathcal{Y}$ for which there is an inner function
$\varphi \in \left \{  g_{1},g_{2},\ldots \right \}  ^{-\alpha}$ such that
$\left \{  g_{1},g_{2},\ldots \right \}  \subset \varphi H^{\alpha}$. It follows
from Lemma \ref{arv} that $\pi \left(  \mathcal{E}\right)  $ is absolutely
measurable and that there is an absolutely measurable cross-section $\rho
:\pi \left(  \mathcal{E}\right)  \rightarrow \mathcal{E}$ such that $\pi \left(
\rho \left(  y\right)  \right)  =y$ for every $y\in \pi \left(  \mathcal{E}%
\right)  $.

We know that $M_{\omega}=0$ if and only if $f_{n}\left(  \omega \right)  =0$
for $n\geq1$. Hence $A=\left \{  \omega \in \Omega:M_{\omega}=0\right \}
=\cap_{n=1}^{\infty}f_{n}^{-1}\left(  \left \{  0\right \}  \right)  $ is
measurable. Let $B=\Omega \backslash A.$ If $\omega \in B,$ then there is an
inner function $\varphi$ such that $M_{\omega}=\varphi H^{\alpha}$. Thus if we
define $\Gamma:B\rightarrow \mathcal{Y}$ by $\Gamma \left(  \omega \right)
=\left(  f_{1}\left(  \omega \right)  ,f_{2}\left(  \omega \right)
,\ldots \right)  $, then $\Gamma \left(  B\right)  \subset \pi \left(
\mathcal{E}\right)  .$ Thus $\rho \circ \Gamma:\Omega \rightarrow \mathcal{X}$ is
measurable, and if we write%
\[
\left(  \rho \circ \Gamma \right)  \left(  \omega \right)  =\left(  g_{1\omega
},g_{2\omega},\ldots,h_{1\omega},h_{2\omega},\ldots,\varphi_{\omega
},n_{1\omega},n_{2\omega},\ldots \right)  ,
\]
we see that $\Phi \left(  \omega \right)  =\varphi_{\omega}$ when $\omega \in B$
and $\Phi \left(  \omega \right)  =0$ when $\omega \in A$ defines the desired
function in $L^{\infty}\left(  \mu,H^{\infty}\right)  $.
\end{proof}

\bigskip

In \cite{RTS} their version of Beurling's theorem was given for the space
$H^{2}\left(  \mathbb{T},\ell^{2}\left(  \mathbb{N}\right)  \right)  ,$ which
is easily seen to be isomorphic to $\ell^{2}\left(  \mathbb{N},H^{2}\left(
\mathbb{T}\right)  \right)  ,$ and the latter is covered by our main theorem.
This raises the question of whether $L^{\beta}\left(  \mu,H^{\alpha}\right)  $
is isometrically isomorphic to $H^{\alpha}\left(  \mathbb{T},L^{\beta}\left(
\mu \right)  \right)  $. If $\alpha=\beta=\left(  \left \Vert \cdot \right \Vert
_{2}+\left \Vert \cdot \right \Vert _{4}\right)  /2$ or if $\alpha=\left \Vert
\cdot \right \Vert _{2}$ and $\beta=\left \Vert \cdot \right \Vert _{4}$, then
these spaces are not isometrically isomorphic, i.e., consider
\[
f\left(  x,z\right)  =\left \{
\begin{array}
[c]{cc}%
1-z & x\in E\\
1-2z & x\in \mathbb{T}\backslash E
\end{array}
\right.  ,
\]
where $\mu=m$ and $\Omega=\mathbb{T}$.

Thus when $\alpha=\beta$ is not a $p$-norm or when $\alpha$ and $\beta$ are
different $p$ -norms, the answer is negative. However, the theorem below shows
that when $\alpha=\beta=\left \Vert \cdot \right \Vert _{p}$ for $1\leq p<\infty
$, then the two spaces are the same.

\begin{proposition}
Suppose $1\leq p<\infty$ and $\alpha=\beta=\left \Vert \cdot \right \Vert _{p}$.
Then $L^{\beta}\left(  \mu,H^{\alpha}\left(  \mathbb{T}\right)  \right)  $ and
$H^{\alpha}\left(  \mathbb{T},L^{\beta}\left(  \mu \right)  \right)  $ are
isometrically isometric.
\end{proposition}

\begin{proof}
Suppose $f=a_{0}+a_{1}z+\cdots+a_{n}z^{n}$ with $a_{0},\ldots,a_{n}\in
L^{\alpha}\left(  \mu \right)  $. We first view $f\in H^{\alpha}\left(
\mathbb{T},X\right)  $. Then we take $\left \vert f\right \vert \left(
z\right)  =\alpha \left(  f\left(  z\right)  \right)  .$ We define
$\beta \left(  f\right)  =\beta \left(  \left \vert f\right \vert \right)  $. We
now consider $f\in L^{\beta}\left(  \mu,H^{\alpha}\left(  \mathbb{T}\right)
\right)  .$ Then $f\left(  \omega \right)  \left(  z\right)  =a_{0}\left(
\omega \right)  +a_{1}\left(  \omega \right)  z+\cdots+a_{n}\left(
\omega \right)  z^{n}$. We then define $\nu:\Omega \rightarrow \lbrack0,\infty)$
by $\nu \left(  \omega \right)  =\alpha \left(  f\left(  \omega \right)  \right)
.$ Then $\beta \left(  f\right)  =\beta \left(  \nu \right)  ,$ and
\[
\alpha \left(  f\right)  ^{p}=\alpha \left(  \beta \left(  f\left(  z\right)
\right)  \right)  ^{p}=\int_{\mathbb{T}}\beta \left(  f\left(  z\right)
\right)  ^{p}dm\left(  z\right)
\]%
\[
=\int_{\mathbb{T}}\left[  \int_{\Omega}\left \vert a_{0}\left(  \omega \right)
+a_{1}\left(  \omega \right)  z+\cdots+a_{n}\left(  \omega \right)
z^{n}\right \vert ^{p}d\mu \left(  \omega \right)  \right]  dm\left(  z\right)
\]%
\[
=\int_{\Omega}\left[  \int_{\mathbb{T}}\left \vert a_{0}\left(  \omega \right)
+a_{1}\left(  \omega \right)  z+\cdots+a_{n}\left(  \omega \right)
z^{n}\right \vert ^{p}dm\left(  z\right)  \right]  d\mu \left(  \omega \right)
\]%
\[
=\int_{\Omega}\nu \left(  \omega \right)  ^{p}d\mu \left(  \omega \right)
=\beta \left(  f\right)  ^{p}.
\]
The functions of the form $f$ above are dense in both $L^{\beta}\left(
\mu,H^{\alpha}\left(  \mathbb{T}\right)  \right)  $ and $H^{\alpha}\left(
\mathbb{T},L^{\beta}\left(  \mu \right)  \right)  $ (see, e.g., Proposition 6.6
in \cite{Chen}); hence, these spaces are isometrically isomorphic.
\end{proof}

\end{document}